\documentclass[11pt]{amsart}
\usepackage[margin=1.25in]{geometry} 

\usepackage{amssymb,enumitem}
\usepackage{amsmath}
\usepackage{amsthm}
\usepackage{mathrsfs}
\usepackage{hyperref}
\usepackage{amsrefs}
\usepackage{comment}
\usepackage{amsthm,amsrefs,amssymb,amsmath}
\numberwithin{equation}{section}
\usepackage{todonotes}
\usepackage{young}
\usepackage{booktabs}
\usepackage{placeins}
\usepackage{amscd,amsmath,amssymb,amsthm,amsfonts,epsfig,graphics}

\newcommand{\Z}{\mathbb{Z}}

\newcommand{\Q}{\mathbb{Q}}
\newcommand{\C}{\mathbb{C}}

\newcommand{\bd}{\begin{description}}
\newcommand{\ed}{\end{description}}

\newcommand{\lcm}{\operatorname{lcm}}
\newcommand{\SL}{\operatorname{SL}}

\newtheorem{theorem}{Theorem}[section]
\newtheorem{cor}[theorem]{Corollary}
\newtheorem{lemma}[theorem]{Lemma}
\newtheorem{prop}[theorem]{Proposition}

\theoremstyle{definition}

\theoremstyle{remark}
\newtheorem{remark}{Remark}[section]

\newtheoremstyle{dotless}{}{}{\itshape}{}{\bfseries}{}{ }{}
\theoremstyle{dotless}

\numberwithin{equation}{section}

\begin{document}
\title{Lacunary Eta-quotients Modulo Powers of Primes}
\author{Tessa Cotron}
\address{Department of Mathematics, Emory University, 201 Dowman Drive, Atlanta, GA 30322}
\email{tessa.cotron@emory.edu}

\author{Anya Michaelsen}
\address{Department of Mathematics, Williams College, 33 Stetson, Williamstown, MA 01267}
\email{anm1@williams.edu}

\author{Emily Stamm}
\address{Department of Mathematics, Vassar College, 124 Raymond Ave, Poughkeepsie, NY 12604}
\email{emily.stamm.12@gmail.com}

\author{Weitao Zhu}
\address{Department of Mathematics, Williams College, 33 Stetson, Williamstown, MA 01267}
\email{wz1@williams.edu}

\begingroup
\def\uppercasenonmath#1{}
\let\MakeUppercase\relax 
\maketitle
\endgroup
\begin{abstract}
An integral power series is called lacunary modulo $M$ if almost all of its coefficients are divisible by $M$. Motivated by the parity problem for the partition function Gordon and Ono studied the generating functions for $t$-regular partitions, and determined conditions for when these functions are lacunary modulo powers of primes. We generalize their results in a number of ways by studying infinite products called Dedekind eta-quotients and generalized Dedekind eta-quotients. We then apply our results to the generating functions for the partition functions considered by Nekrasov, Okounkov, and Han.
\end{abstract}

\section{Introduction}
A \textit{partition} of a positive integer $n$ is a nonincreasing sequence of positive integers whose sum is $n$. For example, the set of partitions of 4 is
\begin{equation*}
\{4,\quad 3+1,\quad 2+2,\quad 2+1+1,\quad 1+1+1+1\}.
\end{equation*}
The partition function $p(n)$ counts the number of partitions of $n$. From the above example we see that $p(4)=5$. The generating function for $p(n)$ satisfies the identity
\begin{equation*}\label{pn}
P(q):=\sum\limits_{n=0}^{\infty}p(n)q^n=\prod\limits_{n=1}^{\infty}\frac{1}{(1-q^n)}.
\end{equation*}
The partition function has many congruence properties modulo primes and powers of primes, the most famous of which are the Ramanujan congruences \cite{ramanujan}:
\begin{align*}
p(5n+4)&\equiv 0\pmod 5,\\
p(7n+5)&\equiv 0\pmod 7,\\
p(11n+6)&\equiv 0\pmod {11}
\end{align*}
for all $n\geq 0$.
Few results of this form were known until Ahlgren and Ono  \cites{ono2,ahlgren} showed that there are infinitely many congruences of the form
\begin{equation*}
p(An+B)\equiv 0\pmod{m},
\end{equation*}
for any integer $m$ relatively prime to 6.

Although no analogous theorem exists for the partition function modulo the primes $2$ and $3$, Parkin and Shanks conjectured that half of the values for $p(n)$ are even and half are odd \cite{parkin}. To be precise, given an integral power series $F(q):=\sum \limits_{n\gg -\infty} a(n)q^n$, we define 
\[
\delta(F,M;X) := \frac{\#\{n\leq X\ :\ a(n)\equiv 0 \pmod{M}\}}{X}.
\]
It is conjectured that $\delta(P,2;X)$ and $\delta(P,3;X)$ tend to $\tfrac{1}{2}$ and $\tfrac{1}{3},$ respectively, as $X$ approaches infinity. The table below contains values of $\delta(P,2;X)$ and $\delta(P,3;X)$ for $X$ up to $500,000$. 
\begin{table}[ht] 
\centering
\begin{tabular}{|c|c|c|}
\hline 
$X$ & $\delta(P,2;X)$ & $\delta(P,3;X)$\\\hline
100,000 & 0.4980$\ldots$  & 0.3334$\ldots$ \\\hline
200,000 & 0.5012$\ldots$  & 0.3332$\ldots$\\\hline
300,000 & 0.5008$\ldots$  & 0.3335$\ldots$\\\hline
400,000 & 0.5000$\ldots$  & 0.3339$\ldots$\\\hline
500,000& 0.5000$\ldots$  & 0.3343$\ldots$\\\hline
\end{tabular}
\vspace{.1 in}
\caption{Data for $p(n)$.}
\end{table}
\FloatBarrier
\noindent
Although calculations strongly support this conjecture, it remains unproven. Moreover, it remains unknown whether a positive proportion of the values of $p(n)$ are even (resp. odd), or whether an infinite number of the values of $p(n)$ lie in any fixed residue class modulo $3$. 

In contrast to the behavior of equal distribution, the series $F(q) = \sum \limits_{n \gg - \infty}a(n)q^n$ is called \textit{lacunary modulo M} if
\begin{equation*}
\lim_{X\rightarrow\infty}\delta(F,M;X)=1,
\end{equation*}
i.e. if almost all of the coefficients of $F$ are divisible by $M$. 
In \cite{gordon}, Gordon and Ono studied the lacunarity of the generating function $G_t(\tau)$ of the $t$-regular partition function $b_t(n)$ defined by
\begin{equation*}
G_t(\tau):=\sum\limits_{n=0}^{\infty}b_t(n)q^n=\prod\limits_{n=1}^{\infty}\frac{(1-q^{tn})}{(1-q^n)}=q^{\frac{1-t}{24}}\frac{\eta(t\tau)}{\eta(\tau)},
\end{equation*}
where $q:=e^{2\pi i\tau}$ and the \textit{Dedekind eta-function} $\eta(\tau)$ is the weight $\tfrac{1}{2}$ modular form 
\begin{equation*}\label{eta}
\eta(\tau):=q^{\frac{1}{24}}\prod\limits_{n=1}^{\infty}(1-q^n)
\end{equation*}
(combinatorially, the function $b_t(n)$ can be interpreted as counting the number of partitions of $n$ with no summands divisible by $t$.)
Gordon and Ono proved that for the finite set of primes $\ell$ such that $\ell^a$  divides $t,$ if $\ell^a\geq \sqrt{t},$ then $G_t(\tau)$ is lacunary modulo $\ell^j$ for any positive integer $j$ \cite{gordon}. 

Similar to $b_t(n)$, the generating functions of many partition functions can be expressed in terms of $\eta(\tau)$. Motivated by this, we study here the lacunarity of eta-quotients of the form
\begin{equation}\label{BigEtaQuot}
G(\tau) :=\frac{\eta(\delta_1\tau)^{r_1}\eta(\delta_2\tau)^{r_2}\cdots\eta(\delta_u\tau)^{r_u}}{\eta(\gamma_1\tau)^{s_1}\eta(\gamma_2\tau)^{s_2}\cdots\eta(\gamma_t\tau)^{s_t}}=q^{\tfrac{E_G}{24}}\sum\limits_{n=0}^{\infty}b(n)q^{n} ,
\end{equation}
where $r_i$, $s_i$, $\delta_i$, and $\gamma_i$ are  positive integers with $\delta_1,\ldots, \delta_u,\gamma_1,\ldots, \gamma_t$ distinct, $u,t\geq 0$, and 
$$E_G: = \sum \limits_{i=1}^u\delta_ir_i-\sum \limits_{i=1}^t\gamma_is_i$$
(note: we will say that $G(\tau)$ is lacunary modulo $M$ when the series $\sum_{n=0}^\infty b(n)q^n$ has that property and similarly for $H(\tau)$ defined below).
Since the eta-function has weight $\tfrac{1}{2}$, the weight of $G(\tau)$ is
\begin{equation*}
k:=\frac{1}{2}\left(\sum_{i=1}^u r_i-\sum_{i=1}^t s_i\right).
\end{equation*}
Define $\mathcal{D}_ G: = \gcd(\delta_1,\ldots, \delta_u)$. 
Generalizing the result from \cite{gordon} stated above, we prove the following theorem.
\begin{theorem}\label{thm1}
Suppose $G(\tau)$ is an eta-quotient of the form (\ref{BigEtaQuot}) with integer weight. If $p$ is a prime such that $p^a$ divides $\mathcal{D}_ G$ and 
\begin{equation*}
p^{a}\geq\;\sqrt{\frac{\sum\limits_{i=1}^t\gamma_i s_i}{\sum\limits_{i=1}^u\frac{r_i}{\delta_i}}},
\end{equation*}
then $G(\tau)$ is lacunary modulo $p^j$ for any positive integer $j$.
Moreover, there exists a positive constant $\alpha$, depending on $p$ and $j$, such that the number of integers $n\leq X$ with $p^j$ not dividing $b(n)$ is  $O\left(\frac{X}{\log^\alpha X}\right)$.
\end{theorem}
\begin{remark}
For primes $p$ that do not satisfy the conditions of Theorem~\ref{thm1}, numerical data (e.g. see Section \ref{exs}) suggest that $1/p$ of the coefficients of $G(\tau)$ are divisible by $p$.
\end{remark}
\begin{remark}
If we apply Theorem~\ref{thm1} to the generating function for the $t$-regular partition function $G_t(\tau),$ we recover the result from \cite{gordon}.
\end{remark}
\begin{remark}
A deep theorem of Serre (see Theorem~\ref{Serre}) asserts that holomorphic integer weight modular forms with integer coefficients are lacunary modulo any positive integer. Since $\eta(\tau)$ is holomorphic and nonvanishing on the complex upper half plane $\mathbb{H}$, the same is true of functions of the form (\ref{BigEtaQuot}); however, these eta-quotients typically have poles at cusps. Therefore, Theorem~\ref{thm1} applies to many functions excluded by Serre's Theorem. 
\end{remark}
Here we give an application of Theorem~\ref{thm1}. We may represent a partition $\lambda_1+\cdots+\lambda_k$ with summands $\lambda_i$ written in nonincreasing order using a Ferrers diagram, which consists of $k$ rows of boxes with the $i$th row containing $\lambda_i$ boxes.
The \textit{hook length} of each box counts the number of boxes to the right of the given box and the number of boxes below the given box, plus $1$ for the box itself. Below we show the Ferrers diagram of $4+2+1$, with each box labeled with its hook length.
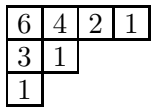
\[
\begin{Young}
6&4&2&1\cr
3&1\cr
1\cr
\end{Young}\]
Denote the set of all partitions as $\mathcal{P}$, the multi-set of hook lengths for $\lambda$ as $\mathcal{H}(\lambda)$, and let 
\begin{equation*}
\mathcal{H}_t(\lambda)=\{h\in \mathcal{H}(\lambda) : h\equiv 0 \pmod{t}\}.
\end{equation*}
Nekrasov and Okounkov \cite{nek-ok} established the identity 
\begin{equation*}
\sum\limits_{\lambda\in\mathcal{P}}q^{|\lambda|}\prod\limits_{h\in\mathcal{H}(\lambda)}\left(1-\frac{z}{h^2}\right)=\prod\limits_{n\geq 1} \left(1-q^n\right)^{z-1},
\end{equation*}
for $z\in \C$. This formula, which has significance in mathematical physics, combinatorics, and number theory, relates partition hook lengths to powers of $\eta(\tau)$. Han \cite{han} provided the $(t,y)$-extension 
\begin{equation*}\label{hansEq}
\sum_{\lambda\in\mathcal{P}}q^{|\lambda|}\prod\limits_{h\in\mathcal{H}_t(\lambda)}\left(y-\frac{tyz}{h^2}\right)=\prod\limits_{n\geq 1}\frac{(1-q^{tn})^t}{(1-(yq^t)^n)^{t-z}(1-q^n)},
\end{equation*}
where $y\in\C$ and $t$ is a positive integer.
One can verify that
\begin{equation*}
G_{1,t,z}(\tau) 
=q^{\tfrac{1-tz}{24}}\frac{\eta^z(t\tau)}{\eta(\tau)}, 
\qquad \text{and}\qquad
G_{-1,t,z}(\tau) 
=q^{\tfrac{1-tz}{24}}\frac{\eta^{2t-z}(t\tau)\eta^{t-z}(4t\tau)}{\eta(\tau)\eta^{3(t-z)}(2t\tau)}
\end{equation*}
(note that $G_{1,t,1}(\tau)$ is the generating function for $b_t(n)$). When $z$ is an odd integer greater than one and $\geq t$,
both $G_{1,t,z}(\tau)$ and $G_{-1,t,z}(\tau)$ are holomorphic positive integral weight modular forms, and hence Serre's theorem gives that these two functions are lacunary modulo any positive integer $M$. 
The following result addresses the case $z<t$.

\begin{cor}\label{cor2}
Suppose $z$ is an odd positive integer with $z<t$, and $p$ is a prime divisor of $t$.\\
1) If $p^a\mid t$ and $$p^a\geq\sqrt{\frac{t}{z}},$$ then $G_{1,t,z}(\tau)$ is lacunary modulo $p^j$ for any positive integer $j$. \\
2) If $p^a\mid t$ and $$p^a\geq 2\;\sqrt{\frac{t+6t^3-6t^2z}{9t-5z}},$$ then $G_{-1,t,z}(\tau)$ is lacunary modulo $p^j$ for any positive integer $j$.
\end{cor}
Theorem~\ref{thm1} is a special case of a more general theorem regarding \textit{generalized eta-functions} (see \cite{robins}). Define the generalized eta-function $\eta_{\delta,g}(\tau)$, with  $\delta,g\in\Z$, by
\begin{equation*}
\eta_{\delta,g}(\tau):=e^{\pi iP_2(\frac{g}{\delta})\delta\tau}\prod_{\substack{n> 0 \\ n\equiv g\pmod{\delta}}}(1-q^{n})\prod_{\substack{n>0 \\ n\equiv -g\pmod{\delta}}}(1-q^n),
\end{equation*}
where $P_2(x):=(x-\lfloor x\rfloor)^2-(x-\lfloor x\rfloor)+\tfrac{1}{6}$ is the second Bernoulli polynomial.  
Note that when $g=0,\frac{\delta}{2},$ we can express the generalized eta-function as the eta-quotient
\begin{equation}\label{relating getEta and regEta}
\eta_{\delta,0}(\tau)=\eta(\delta \tau)^2, \eta_{\delta,\frac{\delta}{2}}(\tau)=\frac{\eta(\frac{\delta}{2}\tau)^2}{\eta(\delta \tau)^2}, 
\end{equation}
respectively. Otherwise the generalized eta-function is a meromorphic modular form of weight zero (i.e. a modular function) on a congruence subgroup, and all generalized eta-functions are holomorphic on $\mathbb{H}.$ 
The relationships in \ref{relating getEta and regEta} as well as the following relationship will be important later on
\begin{equation}\label{n delta n g}
    \eta_{n\delta, ng}(\tau) = \eta_{\delta, g}(n\tau) \text{ for $n \in \Z$.}
\end{equation}

In analogy with eta-quotients, we define the generalized eta-quotient  by
\begin{equation} \label{H}
H(\tau):=\frac{\eta^{r_1}_{\delta_1,g_1}(\tau)\cdots\eta^{r_{u}}_{\delta_{u},g_u}(\tau)}{\eta^{s_1}_{\gamma_1,h_1}(\tau)\cdots\eta^{s_t}_{\gamma_t,h_t}(\tau)} \cdot \frac{\eta^{r_1'}_{\delta_1',\delta_1'/2}(\tau)\cdots\eta^{r_v'}_{\delta_v',\delta_v'/2}(\tau)}{\eta^{s_1'}_{\gamma_1',\gamma_1'/2}(\tau)\cdots\eta^{s_x'}_{\gamma_x',\gamma_x'/2}(\tau)}\cdot \frac{\eta^{r_1''}_{\delta_1'',0}(\tau)\cdots\eta^{r_w''}_{\delta_w'',0}(\tau)}{\eta^{s_1''}_{\gamma_1'',0}(\tau)\cdots\eta^{s_y''}_{\gamma_y'',0}(\tau)},
\end{equation}
where the $r_i$, $r_i'$, $r_i''$, $s_i$, $s_i'$, and $s_i''$ are non-negative and $r_i,s_i\in \Z$, $r_i',r_i'',s_i',s_i''\in\tfrac{1}{2}\Z$.
Furthermore, $h_i, g_i, \gamma_i',\delta_i'\neq 0$, $g_i\neq \delta_i/2$, and $h_i\neq \gamma_i/2$, so that this expression is unique for all generalized eta-functions. 
Then    $H(\tau) = q^{E_H}\sum\limits_{n=0}^{\infty}c(n)q^n$
where 
$${\tiny E_H: = \tfrac{1}{2}\left(\sum_{i=1}^u \delta_iP_2\left(\tfrac{g_i}{\delta_i}\right)r_i - \sum_{i=1}^t \gamma_i P_2\left(\tfrac{h_i}{\gamma_i}\right)s_i - \sum_{i=1}^v\tfrac{\delta_i'r_i'}{12} +\sum_{i=1}^x\tfrac{\gamma_i's_i'}{12} + \sum_{i=1}^w\tfrac{\delta_i''r_i''}{6} - \sum_{i=1}^y\tfrac{\gamma_i''s_i''}{6}\right)}$$

For example, the generalized eta-quotient
\begin{equation*}
\frac{\eta_{5,2}(\tau)}{\eta_{5,1}(\tau)}
= q^{\frac{1}{5}}\prod_{n=0}^{\infty}\frac{(1-q^{5n+1})(1-q^{5n+4})}{(1-q^{5n+2})(1-q^{5n+3})}.
\end{equation*}
is related to the famous Rogers-Ramanujan identities \cite{berndt}*{p. 158} 
\begin{equation*}
\sum_{i=1}^{\infty}\frac{q^{i^2}}{(1-q)\cdots(1-q^i)}=\prod_{n=0}^{\infty}\frac{1}{(1-q^{5n+1})(1-q^{5n+4})} 
\end{equation*}
and
\begin{equation*}
\sum_{i=1}^{\infty}\frac{q^{i^2+i}}{(1-q)\cdots(1-q^i)}=\prod_{n=0}^{\infty}\frac{1}{(1-q^{5n+2})(1-q^{5n+3})}.
\end{equation*}

\begin{remark}
By the identities in (\ref{relating getEta and regEta}) and since  $r_i',r_i''\in\tfrac{1}{2}\Z$ , 
all eta-quotients of the form (\ref{BigEtaQuot}) can be expressed as generalized eta-quotients of the form (\ref{H}). For the rest of this paper, ``eta-quotient'' will refer to those of the form (\ref{BigEtaQuot}) while ``generalized eta-quotient'' will refer to those of the form (\ref{H}).
\end{remark}
\noindent
For such a generalized eta-quotient, define
\begin{equation*}
\mathcal{D} _H: = \gcd(\delta_1,\ldots, \delta_u, \delta_1', \ldots, \delta_v',\delta_1'',\ldots, \delta_w'', \tfrac{\gamma_1'}{2}, \ldots, \tfrac{\gamma_x'}{2}).
\end{equation*}


\begin{theorem}\label{thm3}
Suppose $H(\tau)$ is a generalized eta-quotient of the form (\ref{H}), with
\begin{equation*}
-\frac{1}{2}\sum\limits_{i=1}^u\delta_ir_i+\sum\limits_{i=1}^w\frac{r_i''}{\delta_i''}+\frac{1}{2}\sum\limits_{i=1}^v\frac{r_i'}{\delta_i'}-\frac{1}{2}\sum\limits_{i=1}^x\gamma_i's_i' > 0.
\end{equation*}
If $p$ is a prime divisor of $\mathcal{D}_H$ such that $p^a\mid \mathcal{D}_H$ and
\begin{equation*}
p^a\geq \sqrt{\frac{\sum\limits_{i=1}^t\gamma_is_i +\frac{1}{2}\sum\limits_{i=1}^x\gamma_i's_i'+\sum\limits_{i=1}^y\gamma_i''s_i''+\sum\limits_{i=1}^v\delta_i'r_i'}{-\frac{1}{2}\sum\limits_{i=1}^u\delta_ir_i+\frac{1}{2}\sum\limits_{i=1}^v\frac{r_i'}{\delta_i'}+\sum\limits_{i=1}^w\frac{r_i''}{\delta_i''}-\frac{1}{2}\sum\limits_{i=1}^x\gamma_i's_i'}},
\end{equation*}
then $H(\tau)$ is lacunary modulo $p^j$ for any positive integer $j.$ Moreover, there exists a positive constant $\alpha$ depending on $p$ and $j$ such that the number of integers $n\leq X$ with $p^j$ not dividing $c(n)$ is  $O\left(\frac{X}{\log^\alpha X}\right)$.
\end{theorem}

\begin{remark}
An eta-quotient can be expressed using the generalized eta-quotient form in (\ref{H}) using the $\eta_{\delta_i,0}(\tau)$ or $\eta_{\delta_i, \delta_i/2}(\tau)$ terms (see (\ref{relating getEta and regEta})). If expressed using only the $\eta_{\delta_i,0}$ terms, applying Theorem~\ref{thm3} recovers the results of Theorem~\ref{thm1}. If, however, the expression uses $\eta_{\delta_i, \delta_i/2}(\tau)$ terms, then applying  Theorem~\ref{thm3} gives a weaker bound than the one given by  Theorem~\ref{thm1}. This discrepancy comes from the weaker bounds for the second Bernoulli polynomial $P_2(x)$. 
\end{remark}


In Section~\ref{prelim}, we give preliminaries on modular forms and include a discussion of Serre's theorem on the lacunarity of holomorphic positive integer weight modular forms with integer coefficients. We utilize this theorem in Section~\ref{Proofs1}, where we provide background on Dedekind eta-quotients as well as the proofs of Theorem~\ref{thm1} and Corollary~\ref{cor2}. In Section~\ref{Proofs2} we again make use of the theorem of Serre for the proof of Theorem~\ref{thm3}. Lastly in Section~\ref{exs}, we conclude with specific examples of functions that empirically demonstrate our theorems as well as the potentially slow convergence of $\delta(G,p;X)$ when $G(\tau)$ is lacunary modulo $p$.

\section*{Acknowledgements}
\noindent
The authors would like to thank Ken Ono and Hannah Larson for advising this project, and for their many helpful conversations and suggestions. The authors also thank Emory University, the Templeton World Charity Foundation and the NSF for their support via grant DMS-1557690.

\section{Preliminaries}\label{prelim}

\subsection{Modular Forms}
For $N$ a positive integer, define the level $N$ \textit{congruence subgroups}
$\Gamma_0(N)$ and $\Gamma_1(N)$  of $\SL_2(\mathbb{Z})$ \cite[p. 1]{ono} by 
\begin{align*}
	\Gamma_0(N)&:=\left\{ \begin{pmatrix}
a & b\\
c & d
\end{pmatrix}\in \mathrm{SL}_2(\Z): c\equiv 0 \pmod{N} \right\},\\
\\
\Gamma_1(N)&:=\left\{ \begin{pmatrix}
a & b\\
c & d
\end{pmatrix}\in \mathrm{SL}_2(\Z): c\equiv 0 \pmod{N}, \ a\equiv d\equiv 1 \pmod{N} \right\}.\\
\end{align*}
Let $f(\tau)$ be a meromorphic function on $\mathbb{H}$. For $\gamma=\begin{pmatrix}
a&b\\c&d
\end{pmatrix}\in \SL_2(\Z)$ and  $k\in\Z$ define the weight-$k$ ``slash" operator $|_k$ by 
\begin{equation*}
(f|_k\gamma)(\tau):=(\det{\gamma})^{\frac{k}{2}}(c\tau+d)^{-k}f(\tfrac{a\tau+b}{c\tau+d}).
\end{equation*}
Then $f$ is a \textit{meromorphic modular form of weight k} on a congruence subgroup $\Gamma$, i.e.\textit{ is modular on $\Gamma$}, if (1) for all $\gamma\in \Gamma$ we have $(f|_k\gamma)(\tau)=f(\tau)$, and (2) for any $\gamma\in\SL_2(\Z)$, $(f|_k\gamma)(\tau)$ has a Fourier expansion of the form $$ (f|_k\gamma)(\tau)=\sum\limits_{n\geq n_{\gamma}}a_{\gamma}(n)q_N^n,$$ with $q_N:=e^{2\pi i\tau/N}$ and $a_{\gamma}(n_{\gamma})\neq 0$  \cite[p. 3]{ono}. 

Define a \textit{cusp of $\Gamma$} to be an equivalence class of $\mathbb{P}^1(\Q)=\Q\cup\{\infty\}$ under the action of $\Gamma$. We say that $f(\tau)$ is a \textit{holomorphic modular form} on $\Gamma$ if it is holomorphic on $\mathbb{H}$ and at all the cusps of $\Gamma$ (the latter occurring if $n_{\gamma}\geq 0$ for every $\gamma$). 
The quantity $n_{\gamma}$, called the \textit{order of vanishing at $\frac{-d}{c}$}, depends only on the equivalence class of the cusp $\frac{-d}{c}$. 
The following proposition gives complete sets of representatives for the cusps of $\Gamma_0(N)$ and $\Gamma_1(N)$ (see \cite{diamond}*{p. 99} and \cite{robins}).

\begin{prop}\label{ineqCusps}
Let 
\begin{equation*}\label{cusps}
C_0(N):=\left\{\tfrac{c}{d}: d\mid N, (c,N)=1 \right\},
\end{equation*}
where $c$ runs through a complete residue system modulo $\left(d,\tfrac{N}{d}\right)$, and
\begin{align*}
C_1(N):&=\left\{\tfrac{\lambda}{\mu \epsilon}: \epsilon \mid N, \ 1\leq \lambda,\mu\leq N, \ (\mu,\lambda )=(\lambda, N)=(\mu, N)=1\right\}.
\end{align*}
 Then $C_0(N)$ (resp. $C_1(N)$) is a complete set of representatives of the cusps on $\Gamma_0(N)$ (resp. $\Gamma_1(N)$) Moreover, $C_0(N)$ is minimal.
\end{prop}

\subsection{Serre's Theorem}
The following deep theorem of Serre \cites{serre,mahlburg}, which is proved using the theory of Galois representations, is essential for the proofs of Theorems~\ref{thm1} and~\ref{thm3}.

\begin{theorem}[Serre]\label{Serre}
Let $f(\tau)$ be a holomorphic modular form of positive integer weight with Fourier expansion
$$ f(\tau)=\sum_{n=0}^{\infty}c(n)q^n,$$
where $c(n)$ is an integer for all $n.$ If $M$ is a positive integer, then there exists a positive constant $\alpha$ such that the number of integers $n\leq X$ with $M$ not dividing $c(n)$ is  $O\left(\frac{X}{\log^\alpha X}\right)$.
\end{theorem}

\begin{remark}
Serre proves a more general result addressing the case where $c(n)$ lie in the ring of integers of an algebraic number field. 
\end{remark}

\section{Proof of Theorem~\ref{thm1}}\label{Proofs1}
\noindent
We begin by recalling results regarding Dedekind eta-quotients which will be necessary in the proofs of Theorem~\ref{thm1} and Corollary~\ref{cor2}.

\subsection{Dedekind Eta-quotients}
The following theorems (see \cite{ono}*{p. 18}) give conditions for an eta-quotient to be a modular form on $\Gamma_0(N)$. 
\begin{theorem}\label{ModEta}
If $f(\tau)=\prod\limits_{\delta \mid N}\eta(\delta\tau)^{r_{\delta}}$ is an eta-quotient with $k=\frac{1}{2}\sum\limits_{\delta |N} r_{\delta}\in\Z$ such that
\begin{equation*}
\sum_{\delta \mid N} \delta r_\delta\equiv 0 \pmod{24},
\end{equation*}
and
\begin{equation*}
\sum_{\delta \mid N} \tfrac{N}{\delta} r_\delta \equiv 0 \pmod{24},
\end{equation*}
then 
\begin{equation*}
f\left(\frac{a\tau+b}{c\tau+d}\right)=\chi(d) (c\tau+d)^kf(\tau)
\end{equation*}
for all $\begin{pmatrix}
a & b\\
c & d
\end{pmatrix}\in\Gamma_0(N),$ where $\chi(d):=\left(\frac{(-1)^ks}{d}\right)$ with $s:=\prod\limits_{\delta |N}\delta^{r_{\delta}}$. 
\end{theorem}

\begin{theorem} \label{oov}
Let $c$, $d$, and $N$ be positive integers with $d\mid N$ and $(c,d)=1$. If $f(\tau)$ is an eta-quotient satisfying the conditions of Theorem~\ref{ModEta}, then the order of vanishing of $f(\tau)$ at the cusp $\frac{c}{d}$ is given by
\begin{equation*}\label{order}
\frac{N}{24d\left(d,\frac{N}{d}\right)}\sum\limits_{\delta \mid N}\frac{(d,\delta)^2r_{\delta}}{\delta}.
\end{equation*}
\end{theorem}

 \subsection{Proof of Theorem~\ref{thm1}} \label{pfthm1}
We begin by showing how to construct an integer weight holomorphic modular form that is the product of $G(24\tau)$ and an eta-quotient that is congruent to 1 modulo $p^j$. If $f(x) = 1 + \sum \limits_{n=1}^\infty a(n)x^n$ is an integral power series such that $a(n)\equiv 0 \pmod p$ for all $n \geq 1$, a straightforward induction argument yields that $f^{p^j}(x) \equiv 1 \pmod {p^{j+1}}$ for all $j\geq 0$. For $a\in\Z^+$ define 
\[f_{G,p^a}(\tau):= \prod \limits_{i=1}^t \left(\frac{\eta^{p^{a}}(24\gamma_i \tau)}{\eta(24p^{a} \gamma_i \tau)}\right)^{s_i}=\prod \limits_{i=1}^t \prod \limits_{n=1}^\infty \frac{(1-q^{24\gamma_in})^{p^{a} s_i}}{(1-q^{24p^{a}\gamma_in})^{s_i}}.\]
By the binomial theorem we have $f_{G,p^a}(\tau)\equiv 1\pmod p$, and thus $f_{G,p^a}^{p^j}\equiv 1\pmod{p^{j+1}}$ for all $j\geq 0$. Now define 

\begin{equation*}
F_{G,p^a,j}(\tau):= G(24\tau)f_{G,p^a}^{p^j}(\tau) = 
\frac{\prod \limits_{i=1}^u \eta^{r_i}(24\delta_i\tau)}{\prod \limits_{i=1}^t \eta^{s_i}(24\gamma_i\tau)}
\prod \limits_{i=1}^t \left(\frac{\left(\eta^{p^{a}}(24\gamma_i \tau)\right)}{\left(\eta(24p^{a} \gamma_i \tau)\right)}\right)^{s_ip^j}.
\end{equation*}
Since $f_{G,p^a}^{p^j}\equiv 1\pmod{p^{j+1}}$, we have that
\begin{equation} \label{equiv}
F_{G,p^a,j}(\tau) \equiv G(24\tau) = q^{E_G} \sum \limits_{n=0}^\infty b(n)q^{24n} \pmod {p^{j+1}}.
\end{equation}
Note that by Theorem ~\ref{ModEta}, $F_{G,p^a,j}$ is an eta-quotient of weight $$k_F = \frac{1}{2}\left(\sum \limits_{i=1}^u r_i  - \sum \limits_{i=1}^t s_i + \sum \limits_{i=1}^t s_i p^j(p^{a} -1)\right)$$ on $\Gamma_0(576L^2),$ where 
\begin{equation*}
L_G = \lcm(\delta_1, \cdots, \delta_u, \gamma_1,\cdots, \gamma_t).
\end{equation*}

Since $G(24\tau)$ has integer weight, $$ \frac{1}{2}\left(\sum
\limits_{i=1}^u r_i -\sum \limits_{i=1}^t s_i\right),$$ if $p$ is an odd prime or if $p = 2$ and $j \geq 1,$ then $k_F$ is an integer. Since $p \geq 2$ and  $r_i, s_i \in \Z^+, k_F$ is positive. By Theorem~\ref{ModEta}, this implies that $F_{G,p^a,j}$ is a modular form weight $k_F$ on $\Gamma_0(576L^2)$.
Moreover, since $\eta(\tau)$ does not vanish on the upper half plane, $F_{G,p^a,j}$ is a holomorphic modular form if it is holomorphic at all of the cusps of $\Gamma_0(576L^2).$ It follows from Theorem ~\ref{oov} that $F_{G,p^a,j}$ is holomorphic at a cusp $\frac{c}{d}$ if and only if 
\begin{equation*}
p^{a}\;\frac{\sum\limits_{i=1}^u\dfrac{r_i}{\delta_i}(d,24\delta_i)^2}{\sum\limits_{i=1}^t \dfrac{s_i}{\gamma_i}(d,24p^a\gamma_i)^2}+p^{a}\;\frac{\sum\limits_{i=1}^t \frac{s_i}{\gamma_i}(d,24\gamma_i)^2(p^{a+j}-1)}{\sum\limits_{i=1}^t \dfrac{s_i}{\gamma_i}(d,24p^{a}\gamma_i)^2}\geq p^j,
\end{equation*}
By Proposition~\ref{ineqCusps}, it suffices to check the inequality above for every divisor $d$ of $576L^2$. 

Since $(d,24 p^{a} \gamma _i)^2\leq p^{2a}(d,24 \gamma _i)^2$ for each $\gamma_i$, it is sufficient to show
$$p^{a}\frac{\sum\limits_{i=1}^u\dfrac{r_i}{\delta_i}(d,24\delta _i)^2}{\sum\limits_{i=1}^t\dfrac{s_i}{\gamma _i}(d,24 p^{a} \gamma_i)^2}+p^j-\frac{1}{p^{a}}\geq p^j.$$

\noindent By assumption $p^{a}$ divides every $\delta_i$ so $(d,24p^{a})\leq (d,24\delta_i).$ Since we have $(d,\gamma_i) \leq \gamma_i$, it follows that $F_{G,p^a,j}(\tau)$ is holomorphic at $\frac{c}{d}$ if 
$$p^{a}\frac{\sum\limits_{i=1}^u\tfrac{r_i}{\delta_i}}{\sum\limits_{i=1}^t\tfrac{s_i}{\gamma _i}\gamma_i^2}-\frac{1}{p^{a}}\geq 0,$$

\noindent that is, if
\begin{equation*} \label{p^a}
p^{a} \geq \sqrt{\dfrac{\sum\limits_{i=1}^t s_i\gamma_i}{\sum\limits_{i=1}^u \dfrac{r_i}{\delta_i}}}.
\end{equation*}
By Theorem ~ $\ref{Serre},$ it follows that for $F_{G,p^a,j}(\tau) = \sum b'(n)q^n$ there is a positive constant $\alpha$ such that there are at most $O\left(\frac{X}{log^{\alpha}X}\right)$ integers $n \leq X$ for which $b'(n)$ is not divisible by $p^{j+1}$. 
By (\ref{equiv}), the same holds for $b(n).$
\hfill$\square$
\begin{proof}[Proof of Corollary~\ref{cor2}]
Recall that 
\[G_{1,t,z}(\tau)=q^{\tfrac{1-tz}{24}}\tfrac{\eta^z(t\tau)}{\eta(\tau)}\qquad \text{and}\qquad G_{-1,t,z}(\tau)=q^{\frac{1-tz}{24}}\frac{\eta^{2t-z}(t\tau)\eta^{t-z}(4t\tau)}{\eta(\tau)\eta^{3(t-z)}(2t\tau)},\]
where $z$ is a positive odd integer and $t$ is a positive integer. Thus $\mathcal{D}_{G_{1,t,z}} =\mathcal{D}_{G_{-1,t,z}} = t$.
Applying Theorem~\ref{thm1}, we have that if $p$ is a prime divisor of $t$ such that $p^a\mid t$ and 
$$ p^{a}\geq\; \sqrt{\frac{1}{\frac{z}{t}}}=\sqrt{\frac{t}{z}},$$
then $G_{1,t,z}(\tau)$ is lacunary modulo $p^j$ for any $j\geq 1$ which proves (1).
If $p$ is a prime divisor of $t$ such that $p^a\mid t$ and 
$$p^a\geq \sqrt[]{\frac{1+2t(3(t-z))}{\frac{2t-z}{t}+\frac{t-z}{4t}}}=2 \ \sqrt[]{\frac{t+6t^3-6t^2z}{9t-5z}},$$
then $G_{-1,t,z}(\tau)$ is lacunary modulo $p^j$ for any $j\geq 1,$ completing the proof of (2).
\end{proof}

\section{Proof of Theorem~\ref{thm3}}\label{Proofs2}
\subsection{Generalized Eta-quotients}
We begin by recalling results on generalized eta-quotients analogous to Theorems~\ref{ModEta} and ~\ref{oov}.
\begin{theorem}[{\cite[Theorem 3]{robins}}] \label{generalizedeta}
If $f(\tau)=\prod\limits_{\substack{\delta\mid N\\g}}\eta_{\delta,g}^{r_{\delta,g}}(\tau)$ is such that
\begin{equation*}
\sum\limits_{\substack{\delta\mid N\\g}}\delta P_2\left(\frac{g}{\delta}\right)r_{\delta,g}\equiv 0\pmod 2,
\end{equation*}
and
\begin{equation*}
\sum\limits_{\substack{\delta \mid N\\g}}\frac{N}{6\delta}r_{\delta,g}\equiv 0\pmod 2,
\end{equation*}
then $f(\tau)$ is modular on $\Gamma_1(N)$. 
\end{theorem}
Recall that Proposition~\ref{ineqCusps} gives a complete set of  cusp representatives on $\Gamma_1(N)$. The following theorem gives a formula for computing the orders of generalized eta-quotients at the cusps of $\Gamma_1(N)$.
\begin{theorem}[{\cite[Theorem 4]{robins}}]  \label{general cusp}
Given $f(\tau)$ of the form in Theorem \ref{generalizedeta} and a cusp $\frac{\lambda}{\mu\epsilon}\in C_1(N)$, the order of vanishing of $f(\tau)$ at this cusp in the uniformizing variable $q^{\tfrac{\epsilon}{N}}$is 
\begin{equation*}
\frac{N}{2}\sum\limits_{\substack{\delta \mid N\\g}}\frac{(\delta,\epsilon)^2}{\delta\epsilon}P_2\left(\frac{\lambda g}{(\delta,\epsilon)}\right)r_{\delta,g}.
\end{equation*}
\end{theorem}

\subsection{Proof of Theorem~\ref{thm3}}
Given a generalized eta-quotient $H(\tau)$ of the form (\ref{H}), similar as in the proof of Theorem~\ref{thm1}, we construct an integer weight holomorphic modular form that is the product of $H(\widetilde N\tau)$ and an eta-quotient that is congruent to 1 modulo $p^j$, where $\widetilde N = 24L_H$ and $$L_H:=\lcm (\delta_1,\cdots,\delta_u,\gamma_1,\cdots, \gamma_t\, \delta_1',\cdots, \delta_v', \gamma_1', \cdots, \gamma_x', \delta_1'', \cdots, \delta_w'',\gamma_1'',\cdots,\gamma_y'').$$ Note that for any 
\begin{equation*}
\delta \in \{ \delta_1,\cdots,\delta_u,\gamma_1,\cdots, \gamma_t, \delta_1',\cdots, \delta_v', \gamma_1', \cdots, \gamma_x', \delta_1'', \cdots, \delta_w'',\gamma_1'',\cdots,\gamma_y''\}
\end{equation*}
and $g \in \Z,$ we have $P_2(\tfrac{g}{\delta})\delta \widetilde N  = \left(\left(\frac{l}{\delta}\right)^2 - \frac{l}{\delta} + \frac{1}{6}\right)\delta( 24 \delta)\frac{L_H}{\delta} \in 2\Z.$ For $a\in\Z^+$, define 
\begin{align*}
f_{H,p^a}(\tau):= &\prod_{i=1}^t\left(\frac{\eta^{p^a}_{\gamma_i,0}(\widetilde N\tau)}{\eta_{\gamma_ip^a,0}(\widetilde N\tau)}\right)^{s_i}\prod_{i=1}^v \left(\frac{\eta^{p^a}_{\delta_i',0}(\widetilde N\tau)}{\eta_{\delta_i'p^a,0}(\widetilde N\tau)}\right)^{r_i'}\\
&\prod_{i=1}^x \left(\frac{\eta^{p^a}_{\tfrac{\gamma_i'}{2},0}(\widetilde N\tau)}{\eta_{\tfrac{\gamma_i'p^a}{2},0}(\widetilde N\tau)}\right)^{s_i'}\prod_{i=1}^y \left(\frac{\eta^{p^a}_{\gamma_i'',0}(\widetilde N\tau)}{\eta_{\gamma_i''p^a,0}(\widetilde N\tau)}\right)^{s_i''}.
\end{align*}
By the identities in (\ref{relating getEta and regEta}) and the binomial theorem, $f_{H,p^a}(\tau)\equiv 1\pmod p$, and by an induction argument similar to that in the proof of Theorem~\ref{thm1}, $f_{H,p^a}^{p^j}(\tau)\equiv 1\pmod{p^{j+1}}$ for all $j\geq 0$. Now define
\begin{align*}
F_{H,p^a,j}(\tau):&=H(\widetilde N\tau)f^{p^j}_{H, p^a}(\tau)\\
&=\frac{\prod\limits_{i=1}^u\eta^{r_i}_{\delta_i,g_i}(\widetilde N\tau)}{\prod\limits_{i=1}^t\eta^{s_i}_{\gamma_i,h_i}(\widetilde N\tau)} \cdot \frac{\prod\limits_{i=1}^v\eta^{r_i'}_{\delta_i',\tfrac{\delta_i'}{2}}(\widetilde N\tau)}{\prod\limits_{i=1}^x\eta^{s_i'}_{\gamma_i',\tfrac{\gamma_i'}{2}}(\widetilde N\tau)}\cdot \frac{\prod\limits_{i=1}^w\eta^{r_i''}_{\delta_i'',0}(\widetilde N\tau)}{\prod\limits_{i=1}^y\eta^{s_i''}_{\gamma_i'',0}(\widetilde N\tau)}\cdot f^{p^j}_{H,p^a}(\tau).
\end{align*}
Then
\begin{align*}
F_{H,p^a,j}(\tau) &\equiv H(\widetilde N\tau)
= q^{\widetilde NE_H}\sum c(n)q^{\widetilde Nn} \pmod {p^{j+1}}.
\end{align*}
where $E_H$, defined after (\ref{H}), 
is the order of vanishing at infinity.

Note that by ~(\ref{n delta n g}), we can apply Theorem~\ref{generalizedeta} to $H(\widetilde N\tau)$, which is a modular form of weight $k_H = w - y $ on $\Gamma_1(576L_H^2).$
Theorem~\ref{ModEta} shows that $f^{p^j}_{H, p^a}(\tau)$ is modular on $\Gamma_0(576L_H^2),$ thus modular on $\Gamma_1(576L_H^2).$ Therefore $F_{H,p^a,j}(\tau)$ is modular of weight 
$$k_{F_H}= w - y + \left(\sum \limits_{i=1}^t s_i + \sum\limits_{i=1}^v r_i' + \sum \limits _{i=1}^x s_i'+\sum\limits_{i=1}^y s_i''\right)p^j(p^{a} -1)$$ on $\Gamma_1(576L_H^2).$ 
Using Theorem~\ref{general cusp}, we find that $F_{H,p^a,j}$ is holomorphic at a cusp $\frac{\lambda}{\mu \epsilon}\in C_1(N)$ if and only if
\begin{align*}
&p^{a} \cdot \frac{\sum \limits_{i=1}^u \frac{(\delta_i,\epsilon)^2}{\delta_i}P_2\left(\frac{\lambda g_i}{(\delta_i,\epsilon)}\right)r_i + \sum\limits_{i=1}^w\frac{(\delta_i'',\epsilon)^2}{6\delta_i''} r_i'' + \sum \limits _{i=1}^x\frac{\left(\gamma_i'/2 ,\epsilon\right)^2}{3 \cdot \gamma_i'} s_i'+\sum \limits_{i=1}^v\frac{(\delta_i',\epsilon)^2}{\delta_i'}\left (P_2\left(\frac{\lambda \delta_i'}{2(\delta_i',\epsilon)}\right)+ \tfrac{1}{6}\right)r_i'}{\sum \limits_{i=1}^t \frac{(\gamma_i p^{a},\epsilon)^2}{\gamma_i}s_i + \sum\limits_{i=1}^v\frac{(\delta_i' p^{a},\epsilon)^2}{\delta_i'} r_i' + \sum \limits _{i=1}^x\frac{(\gamma_i'/2 p^{a},\epsilon)^2}{\gamma_i'/2} s_i'+\sum\limits_{i=1}^y\frac{(\gamma_i'' p^{a},\epsilon)^2}{\gamma_i''}s_i''}\\ 
&+ p^{a}\cdot \frac{\sum \limits_{i=1}^t\frac{(\gamma_i,\epsilon)^2}{\gamma_i}\left (\tfrac{1}{6}-P_2\left(\frac{\lambda h_i}{(\gamma_i,\epsilon)}\right)\right)s_i - \sum\limits_{i=1}^x\frac{(\gamma_i',\epsilon)^2}{\gamma_i'}P_2\left(\frac{\lambda \gamma_i'}{2(\gamma_i',\epsilon)}\right)s_i'}{\sum \limits_{i=1}^t \frac{(\gamma_i p^{a},\epsilon)^2}{\gamma_i}s_i + \sum\limits_{i=1}^v\frac{(\delta_i', p^{a},\epsilon)^2}{\delta_i'} r_i' + \sum \limits _{i=1}^x\frac{(\gamma_i'/2 p^{a},\epsilon)^2}{\gamma_i'/2} s_i'+\sum\limits_{i=1}^y\frac{(\gamma_i'' p^{a},\epsilon)^2}{\gamma_i''}s_i''} \\ &+\frac{p^{2a+j}-p^a}{6}\cdot
\frac{\sum\limits_{i=1}^t\frac{(\gamma_i,\epsilon)^2}{\gamma_i}s_i + \sum\limits_{i=1}^v \frac{(\delta_i',\epsilon)^2}{\delta_i'}r_i' + \sum\limits_{i= 1}^x \frac{(\gamma_i'/2,\epsilon)^2}{\gamma_i'/2}s_i' + \sum\limits_{i=1}^y \frac{(\gamma_i'',\epsilon)^2}
{\gamma_i''}s_i''} {\sum \limits_{i=1}^t \frac{(\gamma_i p^{a},\epsilon)^2}{\gamma_i}s_i + \sum\limits_{i=1}^v\frac{(\delta_i' p^{a},\epsilon)^2}{\delta_i'} r_i' + \sum \limits _{i=1}^x\frac{(\gamma_i'/2 p^{a},\epsilon)^2}{\gamma_i'/2} s_i'+\sum\limits_{i=1}^y\frac{(\gamma_i'' p^{a},\epsilon)^2}{\gamma_i''}s_i''} \geq \tfrac{1}{6}p^j.
\end{align*}
Since $-\tfrac{1}{12} \leq P_2\left(x\right) \leq \tfrac{1}{6}$,  $1 \leq (\delta,\epsilon) \leq \delta, \text{ and }\,(\delta p^{a},\epsilon)^2 \leq p^{2a}(\delta,\epsilon)^2,$ this inequality holds if 
$$p^{a}\frac{-\tfrac{1}{12} \sum\limits_{i=1}^u\delta _i r_i -\tfrac{1}{6} \sum\limits_{i=1}^x\tfrac{\gamma_i'}{2}s_i'+ \frac{1}{6}\sum\limits_{i=1}^w\frac{r_i''}{\delta_i''} + \tfrac{1}{12}\sum\limits_{i=1}^v \tfrac{r_i'}{\delta_i'}}{\sum\limits_{i=1}^t\gamma_is_i +\frac{1}{2}\sum\limits_{i=1}^x\gamma_i's_i'+\sum\limits_{i=1}^y\gamma_i''s_i''+\sum\limits_{i=1}^v\delta_i'r_i'} + \left(\tfrac{p^{2a+j}}{6}-\tfrac{p^{a}}{6}\right)p^{-2a}\geq \tfrac{1}{6}p^j,$$
that is, if $$p^a\geq \sqrt[]{\frac{\sum\limits_{i=1}^t\gamma_is_i+\frac{1}{2}\sum\limits_{i=1}^x\gamma_i's_i'+\sum\limits_{i=1}^y\gamma_i''s_i''+\sum\limits_{i=1}^v\delta_i'r_i'}{-\frac{1}{2}\sum\limits_{i=1}^u\delta_ir_i+\frac{1}{2}\sum\limits_{i=1}^v\tfrac{r_i'}{\delta_i'}+\sum\limits_{i=1}^w\frac{r_i''}{\delta_i''}-\frac{1}{2}\sum\limits_{i=1}^x\gamma_i's_i'}}.$$
We apply Theorem~\ref{Serre} to $F_{H,p^a,j}(\tau)$. The remainder of the proof follows as in the proof of Theorem~\ref{thm1}.
\hfill$\square$
\section{Examples} \label{exs}

\subsection{Example of Corollary~\ref{cor2}}
Consider the weight one eta-quotient 
\begin{equation*}
G_{1,18,3}(\tau)= \frac{\eta^3(18\tau)}{\eta(\tau)}.
\end{equation*}
which is not holomorphic at the cusp $\frac{1}{1}$. 
Since $3^2$ divides $18$ and $3^2\geq\sqrt{6}$, by Corollary~\ref{cor2} we have that  $\delta(G_{1,18,3},3; X)\rightarrow1$. In contrast, $2<\sqrt{6}$ and $5\nmid 18$, so the corollary does not apply when $p\in\{2,5\}$. 
The  table below shows $\delta(G_{1,18,3},p;X)$ for $p=2,3,\text{ and }5$ out to the first million coefficients in the power series expansion of $G_{1,18,3}(\tau)$. 

\begin{table}[h] 
\centering
\begin{tabular}{|c|c|c|c|}
\hline 
 $X$ & $\delta(G_{1,18,3},2;X)$ & $\delta(G_{1,18,3},3;X)$ & $\delta(G_{1,18,3},5;X)$\\\hline
   1,000\;  & \ 0.477000  &  \ 0.634000 &  0.1840000 \ \\\hline
     \vdots &   \vdots    &    \vdots   & \vdots     \\\hline 
 200,000\;  & \ 0.498880  & \ 0.687250  & \ 0.199315 \ \\\hline
 400,000\;  & \ 0.498670  & \ 0.693443  & \ 0.199788 \ \\\hline
 600,000\;  & \ 0.499515  & \ 0.696803  & \ 0.200428 \ \\\hline
 800,000\;  & \ 0.500148  & \ 0.699180  & \ 0.200126 \ \\\hline
1,000,000\; & \ 0.500073  & \ 0.701041  & \ 0.200324 \ \\\hline
\end{tabular}
\vspace{.1 in}
\caption{Data for $G_{1,18,3}(\tau)$.}
\end{table}
\FloatBarrier
\subsection{Example of Theorem~\ref{thm3}}
Consider the generalized eta-quotient 
\begin{equation*}
K(\tau) = \frac{\eta_{9,0}(\tau)}{\eta_{6,1}(\tau)}.
\end{equation*}
At the cusp $\tfrac{1}{1}$, $K(\tau)$ is not holomorphic. 
Since $\mathcal{D}_K =9,$ we have that  $p^a\mid \mathcal{D}_K$ for $p=3$, $a=2$. Furthermore, we have that $K(\tau)$ satisfies the inequalities
\begin{equation*}
\frac{r''}{\delta''}=\frac{1}{9}> 0,
\end{equation*}
and
\begin{equation*}
p^a=3^2=9\geq \sqrt{54}= \sqrt{\frac{6\cdot 1}{\tfrac{1}{9}}}= \sqrt{\frac{\gamma s}{\tfrac{r''}{\delta''}}},
\end{equation*}
so by Theorem ~\ref{thm3}, $K(\tau)$ is lacunary modulo $3^j$ for any positive integer $j$. The table below illustrates the extremely slow convergence of $\delta(K,3;X)$ to $1$.
\begin{table}[h] 
\centering
\begin{tabular}{|c|c|c|}
\hline 
$X$  & $\delta(K,3;X)$ & $\delta(K,5;X)$\\\hline
\ 9000   & \ 0.32811  & \ 0.19811 \ \\\hline
\ 18000  & \ 0.33450  & \ 0.19833 \ \\\hline
\ 27000  & \ 0.33919  & \ 0.19907 \ \\\hline
\ 36000  & \ 0.33910  & \ 0.20042 \ \\\hline
\ 45000  & \ 0.34509  & \ 0.20033 \ \\\hline
\end{tabular}
\vspace{.1 in}
\caption{Data for $K(\tau)$.}
\end{table}
\FloatBarrier
\bibliography{references}
 
\end{document}